\documentclass[12pt]{amsart}
\usepackage{etex}
\usepackage{graphicx}
\usepackage{epstopdf}
\usepackage{subcaption}
\usepackage{ytableau}
\usepackage{cite}
\usepackage{tikz, tikz-cd, tkz-graph}
\usetikzlibrary{arrows, patterns}

\usepackage[margin=1in]{geometry}
\usepackage{times}

\usepackage{xcolor}

\usepackage{hyperref}
\hypersetup{
  colorlinks,
  linkcolor={red!50!black},
  citecolor={blue!50!black},
  urlcolor={blue!80!black}
}

\usepackage[english]{babel}
\usepackage{amsmath,amssymb,amsthm}
\usepackage{xypic}
\usepackage{longtable}
\usepackage{array}

 \usepackage[enableskew]{youngtab}

\usepackage{epsfig}
\usepackage{hyperref}
\usepackage{enumitem}
\usepackage{booktabs}

\graphicspath{{./images/}}

\newtheorem{thm}{\bf Theorem}[section]
\newtheorem{eg}[thm]{\bf Example}
\newtheorem{prop}[thm]{\bf Proposition}
\newtheorem{cor}[thm]{\bf Corollary}
\newtheorem{mydef}[thm]{\bf Definition}
\newtheorem{lem}[thm]{\bf Lemma}

\theoremstyle{remark}
\newtheorem{rem}[thm]{\bf Remark}

\newcommand{\ini}{\mathrm{in}}

\newcommand{\CC}{\mathbb{C}}
\newcommand{\RR}{\mathbb{R}}
\newcommand{\ZZ}{\mathbb{Z}}

\newcommand{\NN}{\mathbb{N}}

\newcommand{\br}{\mathbf{r}}

\newcommand{\SSYT}{\mathrm{SSYT}}
\newcommand{\Mon}{\mathrm{Mon}}
\newcommand{\Cone}{\mathrm{Cone}}
\newcommand{\gr}{\mathrm{gr}}
\newcommand{\Exp}{\mathrm{exp}}

\newcommand{\LM}{\mathrm{LM}}

\newcommand{\lcm}{\mathrm{lcm}}

\newcommand{\K}{\mathrm{K}^n_{r_1 r_2}}

\DeclareMathOperator{\Rep}{\mathrm{Rep}}
\DeclareMathOperator{\Spec}{\mathrm{Spec}}
\DeclareMathOperator{\Hom}{\mathrm{Hom}}
\DeclareMathOperator{\GL}{\mathrm{GL}}
\DeclareMathOperator{\Proj}{\mathrm{Proj}}

\DeclareMathOperator{\Gr}{\mathrm{Gr}}
\DeclareMathOperator{\Sym}{\mathrm{Sym}}

\DeclareMathOperator{\Cox}{\mathrm{Cox}}

\DeclareMathOperator{\Mat}{\mathrm{Mat}}
\DeclareMathOperator{\sign}{\mathrm{sign}}
\DeclareMathOperator{\SI}{\mathrm{SI}}
\DeclareMathOperator{\id}{\mathrm{id}}

\newcolumntype{C}[1]{>{\centering\let\newline\\\arraybackslash\hspace{0pt}}m{#1}}

\title[]{A toric degeneration of Kronecker moduli spaces}

\author[E.~Kalashnikov]{E.~Kalashnikov}
\address{Elana Kalashnikov \newline \indent  Department of Pure Mathematics, University of Waterloo,\newline \indent
Waterloo, Canada N2L 3G1}
\email{e2kalash@uwaterloo.ca}

\thanks{The author is supported by an NSERC Discovery Grant. The author thanks Liana Heuberger, Fatemeh Mohammadi, and Oliver Clarke for helpful conversations.}
\begin{document}
\begin{abstract}  In this paper, we show that there is a finite SAGBI basis of the coordinate ring of a Kronecker quiver moduli space, indexed by primitive semi-standard tableaux pairs. This induces a toric degeneration of the Kronecker moduli space to a normal toric variety, a generalization of the toric degeneration of the Grassmannian to the Gelfand--Cetlin polytope constructed by Gonciulea--Lakshmibai \cite{GL}. The moment polytope of the degenerate toric variety can be described as the intersection of two Gelfand--Cetlin polytopes.   We explain when this can be generalized to degenerations coming from matching fields. 
\end{abstract}
\maketitle

\section{Introduction}\label{sec:intro}
The Gelfand--Cetlin toric degeneration of the Grassmannian has been an important construction in mathematics, with generalizations, links, and applications in such diverse areas as Schubert calculus, mirror symmetry, canonical bases, integrable systems and others. This degeneration was constructed by \cite{GL}, and the polytope associated to the degenerate toric variety is the image of the independently constructed Gelfand--Cetlin integrable system \cite{schubert}.  Grassmannians are the simplest family of Kronecker moduli spaces, which are the moduli spaces of (generalized) Kronecker quivers. In this paper, we generalize the Gelfand--Cetlin degeneration to Kronecker moduli spaces.  

The generalized Kronecker quiver (henceforth, just Kronecker quiver) is the quiver with two vertices $1$ and $2$, with $n$ arrows from $1$ to $2$:
 \[\begin{tikzcd}
 1 \arrow[r,"n"] & 2
\end{tikzcd} \]
 Fix a dimension vector $\br=(r_1,r_2).$  Let $G=\GL(r_1) \times \GL(r_2)$, and set  
 \[\theta:=(-r_2,r_1) \in \ZZ^2 \cong \chi(G),\]
 where $\chi(G)$ is the character lattice of $G$. 
 The Kronecker moduli space $\K$ is the projective geometric invariant theory (GIT) quotient 
\[\Mat(r_2 \times r_1)^{\oplus n}/\!/_\theta G,\]
where the group acts by change of basis. When either $r_1$ or $r_2$ is 1, the GIT quotient is a Grassmannian.

The basic steps in constructing the Gelfand--Cetlin toric degeneration of the Grassmannian $\Gr(r,n)$ are as follows. One first fixes a term order on the polynomial ring $\CC[x_{ij}]$, where the $x_{ij}$ are the coordinates on an $r \times n$ matrix, such that the leading monomial $\LM(m)$ of any minor $m$ is the diagonal monomial. The coordinate ring $\CC[\Gr(r,n)]$ of the Grassmannian is a sub-algebra of $\CC[x_{ij}]$. Let $\SSYT(n)$ denote the set of semi-standard Young tableaux with labels in the set $\{1,\dots,n\}$. Then:
\begin{enumerate}
\item Establish bijections between the following three sets:
\[ \{\text{rectangular elements of $\SSYT(n)$ with $r$ rows}\},\]
\[\{\LM(f): f \in \CC[\Gr(r,n)]\},\] 
and a set $\mathcal{B} \subset \CC[\Gr(r,n)]$ which is a vector space basis of the Pl\"ucker algebra.
\item Describe a finite set $S$ in $\SSYT(n)$ such that the associated subset of $\mathcal{B}$ is a SAGBI basis for $\CC[\Gr(r,n)]$; that is, for any $f \in \CC[\Gr(r,n)]$, the leading term of $f$ is a monomial in the leading terms of elements from this subset. 
\item Use the combinatorics of SSYT to describe the toric degeneration and associated polytope. 
\end{enumerate}
In \cite{HeubergerKalashnikov} (joint work with the author here and Heuberger),  the first step above is completed for Kronecker quivers, by introducing \emph{linked tableaux pairs}. In this paper, we show that a finite set, as in the second step, exists. This gives the existence of a finite SAGBI basis, and a toric degeneration. We describe the associated polytope as an intersection of Gelfand--Cetlin polytopes associated to $\Gr(r_1,nr_2)$ and $\Gr(r_2,n r_1)$. We hope that the combinatorics of linked tableaux will allow the third step to be completed more explicitly than the current description.

To state our results more precisely, we recall the definition of linked tableau pairs for $\K$.  We say a partition has shape $p \times q$ if the associated Ferrers' diagram is the  $p \times q$ rectangle.   Let $\alpha^-$ be the partition  of shape $r_1 \times a r_2$, and  $\alpha^+$  the partition of shape $r_2 \times a r_1$ for some $a \in \frac{1}{\gcd(r_1,r_2)} \NN.$ Let $T^-$ be a tableau of shape $\alpha^-$, filled with entries of the form $ij$, where 
 \[i \in \{1,\dots,n\}, \hspace{5mm} j \in \{1,\dots,r_2\}.\]
  Let $T^+$ be a tableau of shape $\alpha^+$, filled with entries of the form $ij$, where 
  \[i \in \{1,\dots,n\}, \hspace{5mm} j \in \{1,\dots,r_1\}.\]
 
The pair $(T^+,T^-)$ form a \emph{linked tableaux pair}  if there is a bijection between the labels of $T^-$ and the labels of $T^+$ taking a label $iq$ in the $p^{th}$ row of $T^-$ to a label $ip$ in the $q^{th}$ row of $T^+.$ We call this bijection the \emph{link} between the pair. Define an ordering on the labels by setting $ip < jq$ if $i <j$ or $i=j$ and $p < q.$ In this way, we extend the notion of semi-standard to these tableaux with double entries: a linked pair is \emph{semi-standard} if both $T^+$ and $T^-$ are semi-standard. A semi-standard linked pair is \emph{primitive}, if it cannot be subdivided into two sets of semi-standard linked pairs (see Definition \ref{def:primitive} for a precise definition). For example, the following is a primitive semi-standard linked pair for $r_1=2,$ $r_2=3$:
\[T^+=\begin{ytableau}
1 1 & 1 1\\
1 2 & 2 1\\
2 2 & 2 2 \\
\end{ytableau}\hspace{8mm} 
T^-=\begin{ytableau}
1 1 & 1 1 & 2 2\\
1 2 & 2 3 & 2 3\\
\end{ytableau}\]

The first result of the paper is the following:
\begin{thm}[see Theorem \ref{thm:finite}] For any $r_1,r_2,$ and $n$, the set of primitive semi-standard tableaux is finite.
\end{thm}
This means that the set of primitive linked tableaux index a SAGBI basis. In this paper, all SAGBI bases are assumed to be finite. As an application, we obtain the following statement:
\begin{thm}[see Theorem \ref{thm:intersection}] There is a flat degeneration of $\K$ to a normal toric variety $X_F$, where $F$ is the normal fan to the polytope $P$ obtained by taking the scaled intersection of two Gelfand--Cetlin polytopes.
\end{thm}
Applying the work of \cite{haradakaveh}, we see that there is an integrable system on $\K$:
\begin{cor}
There is an integrable system on the Kronecker moduli space $\K$, whose image is a scaled intersection of the images of the Gelfand--Cetlin integrable systems of $\Gr(r_1,n r_2)$ and $\Gr(r_2,n r_1)$. \end{cor}

The Gelfand--Cetlin degeneration has been generalized to matching field degenerations  for coherent matching fields\cite{mohammadi1,mohammadi2,mohammadi3}. We generalize the notion of \emph{coherent matching fields} to Kronecker moduli, and give in Theorem \ref{thm:MF} sufficient conditions for a coherent matching field on $\K$ to induce a toric degeneration on $\K$. Under the assumptions of this theorem, the degeneration is a toric degeneration of $\K$ to a normal toric variety $X_F$, where $F$ is the normal fan to a polytope $P$ obtained by taking the scaled intersection of two matching field polytopes. We give a detailed description of a matching field and toric degeneration where the Theorem applies in two examples.

\subsection*{Plan of the paper}
In \S \ref{sec:review}, we review the basic constructions: SAGBI bases,
Kronecker quivers, semi-invariants, and linked tableaux. In \S \ref{sec:finiteness}, we prove that the set of primitive semi-standard tableaux pairs is finite, and describe the two main implications: a toric degeneration and integrable system on $\K$. In \S \ref{sec:matching}, we generalize the notion of matching field degenerations to Kronecker quivers. In the last section of the paper, \S \ref{sec:examples}, we give explicit computations and examples of the results for some small Kronecker moduli.

\section{Semi-invariants and linked pairs of tableaux}\label{sec:review}
\subsection{SAGBI bases and toric degenerations}
In this section, we recall the basics of SAGBI bases and their connection to toric degenerations. SAGBI bases were introduced independently by Robbiano--Sweedler \cite{Robbiano1} (see also \cite{Robbiano2}) and Kapur--Madlener \cite{deepak}. 

On polynomial rings $\CC[x_1,\dots,x_n]$ we consider orders given by weights $w \in \RR^n_{\geq 0}$. Any such $w$ assigns a weight $w_i$ to each $x_i$, and a monomial $\prod_{i=1}^n x_i^{a_i}$ is given the weight $\sum_{i=1}^n a_i w_i.$ This allows us to define a weak order on the monomials in the polynomial ring given by their weights. Given a polynomial $f$, $\ini_<(f)$ is the polynomial made up of the terms of $f$ that achieve the maximum weight among the monomials of $f$. In most cases that we consider, this will be a monomial, which we call the initial term.  
\begin{mydef} Let $A \subset \CC[x_1,\dots,x_n]$ be a sub-algebra, together with a choice of weight $w$.  Then a set $f_1,\dots,f_k$ is a \emph{SAGBI} basis if for each $f_i$, $\ini_<(f_i)$ is a monomial and
 \[\ini_<(A)=\CC[\ini_<(f_1),\dots,\ini_<(f_k)].\]
\end{mydef}
One reason to be interested in SAGBI bases is their role in toric degenerations. Let $X$ be a projective variety $\Proj(R)$, for $R$ a graded sub-algebra of a sufficiently large polynomial ring $\CC[x_1,\dots,x_N]$. Suppose that $f_1,\dots,f_k$ is a finite SAGBI basis for $R$ for some choice of weight. Then $\ini_<(R)$ is generated by the $\ini_<(f_i)$ (in particular, it is finitely generated). There is a flat family of algebras $\mathcal{R} \to \CC$ such that $\mathcal{R}_1=R$ and $\mathcal{R}_0=\ini_<(R)$, obtained by deforming the polynomials in $R$ to their initial terms \cite{eisenbud}.

The initial algebra $\ini_<(R)$ is the coordinate ring of the (possibly) non-normal toric variety $\Proj(\ini_<(R))$. We will be most interested in the case where the semi-group  algebra $\ini_<(R)$ is the semi-group algebra associated to a rational cone. Let $\Exp(m)$ of a monomial $m \in \CC[x_1,\dots,x_N]$ be the exponent vector. We take the cone $C$ over the elements
\[\{ (\Exp(\ini_<(f_i)),\gr(f_i)) \in \RR^N \times \RR: i=1,\dots,k\}.\]
Let $L$ be the standard lattice in $\RR^N \times \RR.$ We \emph{assume} that 
\begin{equation} \label{normal} C \cap L=\{(\Exp(m),\gr(m)): m \text{ a monomial in } \ini_<(R)\}.\end{equation}
That is, suppose that the lattice points of the cone over the exponents of the $\ini_<(f_i)$ are exactly the set of exponents of monomials in $\ini_<(A)$. Then $\Spec(\ini_<(R))$ is a normal toric variety. 

\begin{thm} In the above set up,  there is a flat degeneration of $X$ to $\Proj(\ini_<(R))$.   
\end{thm}
This appears in various places in the literature; some relevant (but not exhaustive) citations include \cite[Theorem 15.17]{eisenbud} (flatness, including in the case where $\ini_<$ is not a total order) and \cite{sturmfels}. In \cite{kaveh} there is a nice description of the projective version. 

We now identify the projective toric variety $\Proj(\ini_<(R))$.  We can choose a large enough $m$ such that $P=C \cap \RR^N \times \{m\}$ is a normal lattice polytope. Let $F$ be the normal fan of $P$; if $P$ is not full-dimensional, we take the normal fan of $P-x$ in the real vector subspace given by the span of $P-x$ for any lattice point $x$ in $P$. Then we have the following identification: 

 \begin{cor}\label{cor:toric} $\Proj(\ini_<(R))=X_F$
 \end{cor}
If we had not assumed  \eqref{normal}, then $X_F$ would be the normalization of $\Proj(\ini_<(R)).$ This is a standard result in toric geometry, but is specifically addressed in the discussion following Proposition 3.4 in \cite{Feigin} in the situation where the cone is not necessarily full dimensional. 
 
 \subsection{Construction of Kronecker moduli}
 This paper is interested in quiver moduli spaces given by \emph{Kronecker quivers}. These are quivers with two vertices, and $n$ arrows between them: 
\[\begin{tikzcd}
 1 \arrow[r,"n"] & 2
\end{tikzcd} \]
We denote the quiver with $n$ arrows as $Q^n$, and set $Q^n_0=\{1,2\}$ to be the set of vertices, and $Q^n_1$ the set of arrows. There are maps $s,t:Q^n_1 \to Q^n_0$ taking arrows to their source and target vertices respectively. 

To describe a moduli problem, we assign dimensions to each vertex; that is, we fix \[\br:=(r_1,r_2) \in \NN^2.\]  The set of representations of the quiver of dimension $\br$ is
\[\Rep(Q^n,\br):=\Hom_{a \in Q^n_1}(\CC^{r_{s(a)}}, \CC^{r_{t(a)}})=\Mat(r_2 \times r_1)^{\oplus n}.\]
We will need coordinates on this space: we let 
\begin{equation} \label{def:Ai} (x^i_{jk})_{\substack{1 \leq j \leq r_2,\\1 \leq k \leq r_2}}\end{equation}
be coordinates on the $i^{th}$ matrix $A_i$ in the $n$-tuple of matrices $A_1,\dots,A_n$. 

The group $G:=\GL(r_1) \times \GL(r_2)$ acts by change of basis on $\Rep(Q^n,\br)$. The character lattice of $G$, denoted $\chi(G)$, can be identified with $\ZZ^2$, where the standard generators are given by the two determinant maps.  Let $\lambda_0:=\CC^* \to G$ be the diagonal one parameter subgroup. Then the image of $\lambda_0$ acts trivially on $G$. For any character $\theta \in \chi(G)$ such that $\langle \theta,\lambda_0 \rangle =0$, we can define the GIT quotient
\[M_\theta(Q^n,\br)=\Rep(Q^n,\br)/\!/_\theta G.\]

As it turns out, for the Kronecker quiver, there is only one non-trivial GIT quotient, induced by setting $\theta$ to be any positive multiple of $(-r_2,r_1)$. For such a choice of $\theta,$ we set 
\[\K:=M_\theta(Q^n,\br).\]

\begin{thm}\cite{reineke, fanoquiver} When $\gcd(r_1,r_2)=1$, $n \geq 3$, $\K$ is a smooth Fano variety of Picard rank one and Fano index $n$. \end{thm}

By definition, the GIT quotient $\Rep(Q^n,\br)/\!/_\theta G$ is \[\Proj(\oplus_{k \in \NN} \SI^G_{k \theta}(Q^n,\br)),\] where $ \SI^G_{k \theta}(Q^n,\br)$ is the set of semi-invariants of $Q^n$ of weight $k \theta \in \chi(G)$. That is,  $ \SI^G_{k \theta}(Q^n,\br)$ is the set of polynomials $f \in   \CC[\Rep(Q^n,\br)]$ satisfying
\[ f(g \cdot v)=(k \theta)(g) f(v),\]
 for all $v \in \Rep(Q^n,\br)$ and $g \in G$.
 
 We define the semi-invariant ring to be
\[\SI^G(Q^n,\br)=\oplus_{\theta \in \chi(G)} \SI^G_\theta(Q^n,\br).\] 

Given $\alpha \in \chi(G)$ satisfying $\langle \alpha,\lambda_0 \rangle=0$, and $f \in \SI^G_\alpha(Q^n,\br)$, the trivial line bundle 
\[ \Rep(Q^n,\br)^{ss} \times \CC \to \Rep(Q^n,\br)^{ss}\]
descends to a line bundle $L_\alpha$ on $M_\theta(Q,\br)$, and $f$ descends to a global section of $L_\alpha$. Roughly speaking, the Cox ring of a variety $X$ is the ring of all global sections of all line bundles on $X$. Therefore, we have described a morphism
\[\SI^G(Q^n,\br) \to \Cox(\K).\]
By \cite[Proposition 5.2]{reineke} and the discussion in \S 5.2 of \cite{fanoquiver}, if $\gcd(r_1,r_2)=1$, the Kronecker moduli space $\K$ satisfies the conditions of \cite[Lemma 3.3]{fanoquiver} and so we can conclude that
\begin{prop} The morphism 
\[\SI^G(Q^n,\br) \to \Cox(\K)\]
is an isomorphism.
\end{prop}
In \cite{crawunpub}, the authors also discuss this question (more generally than just for Kronecker moduli). 

\subsection{Semi-invariants of Kronecker moduli}
In this section, we recall the results of \cite{HeubergerKalashnikov} on the semi-invariant ring of Kronecker moduli spaces. We first describe how to produce from a \emph{linked pair of tableaux} a Domokos--Zubkov semi-invariant \cite{domokos}. 

Fix a Kronecker moduli space, $\K$, and a character $\alpha \in \chi(G)$ satisfying $\langle \alpha,\lambda_0 \rangle =0.$ That is, if $\alpha=(-\alpha_1,\alpha_2)$, then 
\[r_1 \alpha_1 = r_2 \alpha_2.\]
Unless $\alpha_i>0$, there are no non-trivial semi-invariants, so we assume this. 

Recall from the introduction that a linked pair of tableaux  $(T^+,T^-,\sigma)$ of weight $\alpha$ is the data:
\begin{itemize}
\item a tableaux $T^-$ of size $r_1 \times \alpha_1$, with labels $ik$, where $1\leq i \leq n$, and $1 \leq k \leq r_2,$
\item a tableaux $T^+$ of size $r_2 \times \alpha_2$, with labels $ij$, where $1\leq i \leq n$, and $1 \leq j \leq r_1,$
\item a bijection $\sigma$ taking the labels of $T^-$ to the labels of $T^-$, satisfying the condition that a label $ik$ in row $j$ of $T^-$ is mapped to a label $ij$ in the $k^{th}$ row of $T^+.$
\end{itemize}
To either $T^+$ or $T^-$, we can associate a monomial:
\[\Mon(T^+):=\prod_{k=1}^{r_2} \prod_{\substack{\text{labels } ij \\ \text{ in row $k$}}} x^i_{kj},\]
\[\Mon(T^-):=\prod_{j=1}^{r_1}  \prod_{\substack{\text{labels } ik \\ \text{ in row $j$}}} x^i_{kj}.\]
The link ensures that $\Mon(T^+)=\Mon(T^-)$.

The Weyl group of $G$ is $W=\Sym(r_1) \times \Sym(r_2)$. Set $W^\alpha=\Sym(r_1)^{\alpha_1} \times \Sym(r_2)^{\alpha_2}$. We define an action of $W^\alpha$ on $T^+$:
\begin{itemize}
\item The factor $\Sym(r_2)^{\alpha_2}$ acts by permuting each of the $\alpha_2$ columns of $T^+$. 
\item The factor $\Sym(r_1)^{\alpha_1}$ acts on $T^-$ by permuting each of the $\alpha_1$ columns of $T^-$. Using the link $\sigma,$ this induces an action on the second components of the labels of $T^+$, which is the required action of $\Sym(r_1)^{\alpha_1}$ on $T^+$.  
\end{itemize}
There is an analogous action of $W^\alpha$ on $T^-$. 
\begin{mydef} The semi-invariant of weight $\alpha$ associated to the linked tableaux pair $(T^+,T^-,\sigma)$ is
\[f_{T^\pm}:=\sum_{\sigma \in W^\alpha} \sign(\alpha) \Mon(\alpha \cdot T^+)=\sum_{\sigma \in W^\alpha} \sign(\alpha) \Mon(\alpha \cdot T^-).\]
\end{mydef}
Recall that the $A_i:=[x^i_{jk}]$ are the coordinate matrices defined in \eqref{def:Ai}. 
Unless otherwise stated, we will let $\LM(f)$ denote the leading monomial of an element of the semi-invariant ring, under any term order that is \emph{diagonal}. That is, we require that the leading monomial of any maximal minor of either
\[\begin{bmatrix} A_1 \mid & A_2 \mid &\cdots & A_n\end{bmatrix}\]
or 
\[\begin{bmatrix} A_1 \\
A_2\\
\vdots\\
A_n \end{bmatrix}\]
is the diagonal term. One such term order is given by taking $x^i_{jk}>x^l_{pq}$ when $i<l$, and within each $A_i$, taking the order defined by going across rows left to right and top to bottom.

A linked pair $(T^\pm,\sigma)$ is \emph{semi-standard} if both $T^+$ and $T^-$ are semi-standard tableaux, using the lex ordering on the double labels.
\begin{thm}\cite{HeubergerKalashnikov} Let $f$ be a semi-invariant of $\K$. Then for some semi-standard pair $T^\pm$,
\[\LM(f)=\Mon(T^+)=\LM(f_{T^\pm}).\]
The semi-invariants $f_{T^\pm}$ for $T^\pm$ semi-standard span (as a vector space) the semi-invariant ring.
\end{thm}
We recall one final definition:
\begin{mydef}\label{def:primitive} Let $T^\pm$ be a semi-standard linked pair. Then $T^\pm$ is \emph{primitive} if there does not exist semi-standard pairs $T_1^\pm$ and $T_2^\pm$ such that
\[\Mon(T^+)=\Mon(T^+_1)\Mon(T^+_2).\]
\end{mydef}
The main goal of this paper is to show that there are only finitely many primitive semi-standard tableaux pairs.
\section{The main theorem}\label{sec:finiteness}
The first step is to translate the problem above into a question of finite generation of cones. Fixing a Kronecker moduli space $\K$, we consider the lattice $M=\ZZ^{nr_1 r_2}$, which we identify with $\Mat(r_2 \times r_1,\ZZ)^n.$ To a linked tableaux pair $T^\pm$, we can associate an element of the lattice,
$v_{T^\pm}$, which is the exponent vector of $\Mon(T^+).$ 
\begin{lem} Let $T^\pm_1$ and $T^\pm_2$ be two semi-standard linked pairs. If $v_{T^\pm_1}=v_{T^\pm_2}$, then $T^\pm_1=T^\pm_2.$
\end{lem}
\begin{proof}
From $v_{T^\pm_1}$, the monomial is determined, and in particular the labels contained in each row of $T^\pm_2.$ However, since the rows of $T^\pm_2$ are weakly increasing, this in fact determines the tableaux. 
\end{proof}
\begin{mydef} For a Kronecker moduli space $\K$, we define the cone
\[C^n_{r_1 r_2}:=\Cone(v_{T^\pm}: T^\pm \text{ is a semi-standard pair}).\]
\end{mydef}
The semi-invariant ring is naturally graded by $\ZZ \cong \{\alpha \in \chi(G): \langle \alpha,\lambda_0 \rangle=0\},$ where the identification is induced by the identification $\chi(G) \cong \ZZ^2$ and is given by
\[k \mapsto {k}{\gcd(r_1,r_2)}(-r_2,r_1) \in \chi(G).\]
This gives a height function on $C^n_{r_1 r_2}$, which can also be written as 
\[(v^i_{jk}) \mapsto \frac{1}{\lcm(r_1,r_2)} \sum_{i,j,k} v^i_{jk}.\]
Unless otherwise stated, by height we always mean this height function. 
\begin{mydef} We define the polytope (which may not be a lattice polytope) $P^n_{r_1 r_2}(k)$ to be the height $k$ slice of $C^n_{r_1 r_2}$. 
\end{mydef}

\begin{thm}\label{thm:finite} The cone $C^n_{r_1 r_2}$ is a finitely generated cone, and 
\[C^n_{r_1 r_2} \cap M=\{v_{T^\pm}: T^\pm \text{ is a semi-standard pair}\}\]
\end{thm}
\begin{proof} We prove the theorem by showing that $C^n_{r_1 r_2}$ is the intersection of a finite number of half spaces of the form 
\[H_y:=\{x \in M: y(x) \geq 0\}\]
for some $y \in M^\vee.$ 

We first note that every $v_{T^\pm}$ has non-negative entries; this condition can clearly be ensured using half spaces. Consider a lattice vector $v=(v^i_{jk})$ in the positive orthant. If $v=v_{T^\pm}$ for some $T^\pm$, then for each $j_1$ and $j_2$ between $1,\dots,r_2,$
\begin{equation}\label{eq:sum1} \sum_{i,k} v^i_{j_1 k}=\sum_{i,k} v^i_{j_2 k},\end{equation}
and for each $k_1$ and $k_2$ between $1,\dots,r_1,$
\begin{equation}\label{eq:sum2} \sum_{i,j} v^i_{j k_1}=\sum_{i,j} v^i_{j k_2}.\end{equation}
These constraints can can be ensured using half spaces (equality is achieved by taking both $H_y$ and $H_{-y}$ for appropriate $y \in M^\vee$).

We now consider any $v$ satisfying all the constraints mentioned so far. To such a $v$ we associate a linked tableaux pair $T^\pm$, that both have weakly increasing rows. The $j^{th}$ row of $T^+$ has $v^i_{jk}$ labels of the form $ik$, and the $k^{th}$ row of $T^-$ has $v^i_{jk}$ labels $ij.$ There is clearly only one such linked pair. We will now give hyperplane conditions that exactly ensure that $T^\pm$ are both semi-standard. 

To complete the proof, we show the following claim, and its analogue for $T^-$: $T^+$ is semi-standard if, for every two consecutive rows in $T^+$, say row $j$ and $j+1$, and every choice of $i,k$,
\begin{equation} \label{eq:inequalities} \sum_{sk \leq ik-1} v^s_{j l} \geq \sum_{sk \leq ik} v^s_{j l}.\end{equation}

The claim holds for the following reason. Consider the last label $ik$ in the $j+1$ row. This is the $\sum_{sk \leq ik} v^s_{j l}$ entry in this row. Then if $T^+$ is semi-standard, the label above $ik$ is at most $i k-1$. This means that the number of labels in row $j$ of form $sk \leq i k-1$ is at least $\sum_{sk \leq ik} v^s_{j l}$; that is, 
\[\sum_{sl \leq ik-1} v^s_{j l} \geq \sum_{sl \leq ik} v^s_{j l}.\]
Thus, every semi-standard tableaux pair satisfies the inequalities \eqref{eq:inequalities}. We conclude that $C^n_{r_1 r_2}$ is the intersection of a finite number of half spaces, and hence is finitely generated by Gordan's Lemma.

We now show the second part of the theorem. Clearly every $v_{T^\pm}$ for a semi-standard pair is an element of $C^n_{r_1 r_2}$.  Conversely, suppose that $v \in C$ is a lattice point. By equations \eqref{eq:sum1}  and \eqref{eq:sum2}, we can construct a linked tableaux pair $T^\pm$ with $v=v_{T^\pm}$ such that both tableaux have weakly increasing rows.  Consider a label $ik$ appearing in row $j+1$ of $T^+$. The inequalities imply if $sl$ is the label above $ik$ in $T^+$, then $sl \leq ik-1<ik$. Thus $T^+$ is semi-standard as required.

Clearly, each of the inequalities \eqref{eq:inequalities} describe a halfspace $H_y$ for suitable $y$, and the same holds for their $T^-$ analogue. \end{proof}
\begin{cor} There are only finitely many primitive semi-standard linked tableaux pairs.
\end{cor}
\begin{proof} As $C^n_{r_1 r_2}$ is finitely generated, there is a finite set $v_1,\dots,v_k$ of generators. If $T^\pm$ is primitive, then $v_{T^\pm}$ must be a generator, as by definition, we cannot write $v_{T^\pm}$ as a sum of any other two elements in the cone. Therefore, there are only finitely many primitive tableaux pairs. 
\end{proof}
\begin{rem} In \cite{HeubergerKalashnikov}, we proved combinatorially -- and laboriously -- the above corollary in the $r_1=r_2=2$ case. Unlike the above corollary, however, this gives an explicit description of the SAGBI basis. 
\end{rem}
\begin{eg} For non-trivial examples, see \S \ref{sec:examples}. The simplest example, however, of $C^n_{r_1 r_2}$, is when either $r_1=1$ or $r_2=1$. The cases are isomorphic, we so can consider the cone $C^n_{1 r}$, and its slice at height one, $P_{r,n}:=P^n_{1 r}(1)$. In this case, $C^n_{1 r}$ is generated at height one, and $P_{r,n}$  is the Gelfand--Cetlin polytope (or isomorphic to it, depending on your description). In \cite{haradakaveh}, they show that there the integrable system can be obtained by composing the map they produce from  $\Gr(r,n) \to X_{r,n}$ with the integrable system on $X_{r,n}$ using its toric structure. As, by definition, $X_{r,n}$ is $\Proj$ of the semigroup algebra obtained from $C^n_{1 r}$, the image of its integrable system is naturally seen as $P_{r,n}$. 
\end{eg}

\begin{cor} The finite set $\{f_{T^\pm}: T \text{ primitive semi-standard tableaux pair}\}$ is a SAGBI basis for the Cox ring of $\K$, when $\gcd(r_1,r_2)=1.$
\end{cor}

SAGBI bases are of interest in part because of their close connection to toric degenerations. We now describe the resulting toric degeneration of $\K$. 

In the next theorem, we show that $\K$ has a toric degeneration to the toric variety associated to the intersection of (scaled) Gelfand--Cetlin polytopes $P_{r_1,nr_2}$ and $P_{r_2,n r_1}$. The lattice polytope $P_{r_1,n r_2}$ lives in the toric lattice $M_1:=\Mat(r_1,n r_2;\ZZ)$, and the lattice polytope $P_{r_2,n r_1}$ lives in the lattice  $M_2:=\Mat(r_2,n r_1;\ZZ)$. We identify both of these lattices with $M =\Mat(r_2,r_1;\ZZ)^{\times n}$, by taking a tuple of matrices in $M$ and either stacking them horizontally (giving an element of $M_2$) or vertically and then transposing (giving an element of $M_1$), so that taking an intersection makes sense. 

\begin{thm}\label{thm:intersection} Assume $\gcd(r_1,r_2)=1$. Then the SAGBI basis given by primitive semi-standard tableaux pair induces a toric degeneration of $\K$ to a normal toric variety $X_F$, where $F$ is is normal fan of 
\[ r_2 P_{r_1,n r_2} \cap r_1 P_{r_2,n r_1}.\]
\end{thm}
\begin{proof}
We  claim that 
\[P^n_{r_1 r_2}:=r_2 P_{r_1,n r2} \cap r_1 P_{r_2,n r_1}.\]
Note that $C^{n r_2}_{1 r_1}$ is the cone over $P_{r_1,n r2}$ and  $C^{n r_1}_{1 r_2}$ is the cone over $P_{r_2,n r_1}$. The claim then follows if we can show that
\[C^n_{r_1 r_2} = C^{n r_1}_{1 r_2} \cap C^{n r_2}_{1 r_1}.\]
A lattice point is an element of the right hand side if and only it is the exponent vector of a monomial $m$ that can be represented at $\Mon(T^+)$ and $\Mon(T^-)$ for two semi-standard tableau $T^+$ and $T^-$. These tableaux are then automatically linked (using the monomial), and hence the exponent vector is an element of $C^n_{r_1 r_2}$. Therefore,
\[P^n_{r_1 r_2}=r_2 P_{r_1,n r2} \cap r_1 P_{r_2,n r_1}\]
as desired.  The rest of the claim follows from Corollary \ref{cor:toric}. 
\end{proof}
In general, $P^n_{r_1,r_2}$ is not a lattice polytope, though there is some smallest positive integer $k$ such that  $k P^n_{r_1,r_2}$ is a lattice polytope. This corresponds to the fact that the SAGBI basis is not generated in degree 1. It would be interesting to determine bounds on $k$. 
\begin{mydef} We call this toric degeneration the \emph{Gelfand--Cetlin toric degeneration} of $\K$. 
\end{mydef}
Theorem \ref{thm:intersection} shows that, given the Gelfand--Cetlin toric degenerations of two Grassmannians, $\Gr(r_1,nr_2)$ and $\Gr(r_2,n r_1)$, we can intersect the polytopes of each to produce a toric degeneration of a Kronecker moduli space. Toric degenerations of Grassmannians are well-studied with considerable combinatorial structure, so a natural question is the extent to which this can be generalized to other degenerations.  One key aspect is that the polytopes associated to the degenerations live in $M$. In the next section, we give conditions for when the intersection construction above can be generalized to toric degenerations coming from matching fields. 
\subsection{Applications to mirror symmetry}
One application of toric degenerations is in finding \emph{mirrors} to Fano varieties. We will discuss this application in examples in the last section. For some background on mirror symmetry for Fano varieties, see \cite{fanomanifolds, CoatesCortiGalkinKasprzyk2016}. The mirror to a deformation class of $n$-dimensional smooth Fano varieties, also known as its (weak) \emph{Landau--Ginzburg model}, is a  mutation class of certain Laurent polynomials in $n$ variables. The Laurent polynomials that are mirror to smooth Fano varieties are, conjecturally, \emph{rigid maximally mutable} Laurent polynomials \cite{Coates_2021}. 

To check that representatives $X$ and $f$ belong to mirror classes, one shows that two power series coincide. The power series associated to a Fano variety $X$ is the \emph{quantum period}, which is built out of genus zero Gromov--Witten invariants, and hence deformation invariant. The power series associated to the Laurent polynomial $f$ is called the \emph{classical period}; it is easy to compute and is mutation invariant. See \cite{CoatesCortiGalkinKasprzyk2016} for a definition of both periods. Checking that the two periods coincide can be very difficult, as it is often hard or impossible to find a closed formula for the quantum period of a Fano variety. For Fano toric complete intersections, this is a consequence of the celebrated mirror theorem \cite{givental, lian}.  

The Newton polytope $P$ of a Laurent polynomial $f$ which is a \emph{rigid maximally-mutable} Laurent polynomial \cite{Coates_2021} is a Fano polytope, i.e. $P$ spans the fan of a (singular) toric Fano variety $X_P$. Part of the mirror symmetry conjectures is that $X_P$ can smooth to a Fano variety $X$ which is mirror to $f$. Thus, if we want to find a conjectural mirror to a Fano variety $X$, we find a toric degeneration of $X$ to some $X_P$. This determines the monomials of $f$, and coefficients are then chosen to ensure that $f$ is rigid maximally mutable. When $P$ is terminal (i.e. the only lattice points of $P$ are the origin and the vertices, see \cite{kasprzyk2012fano}), the associated Laurent polynomial is just the one with coefficient $1$ assigned to each vertex.  

\section{Matching Fields for Kronecker quiver moduli} \label{sec:matching}
\subsection{Matching fields for Grassmannians}
Matching fields and associated toric degenerations have been studied by many authors \cite{mohammadi1, mohammadi2,mohammadi3, Higashitani}. We give a brief introduction here.

Consider the Grassmannian $\Gr(r,n)$. Its Cox ring is generated by Pl\"ucker coordinates, $p_{I}$ where $I=\{i_1,\dots,i_r\}$ is a size $r$ subset of  $\{1,\dots,n\}$. Then $p_I$ is a signed sum of monomials of the form 
\[x_{i_1 \sigma(1)} \cdots x_{i_r \sigma(r)}\]
ranging over permutation $\sigma \in \Sym(r)$, where $[x_{ij}]$ are the coordinates on an $r \times n$ matrix. 

The Gelfand--Cetlin toric degeneration picks out the monomial corresponding to the identity permutation for each $p_I$. A \emph{matching field} for $\Gr(r,n)$ is a map $\Lambda$
\[\Lambda: \{ I \subset \{1,\dots,n\} | |I|=r\} \to \Sym(r), I \mapsto \sigma_I\]
One can think of this as picking out a monomial out of the polynomial of each Pl\"ucker coordinate:
\[p_I \mapsto x_{i_1 \sigma_I(1)} \cdots x_{i_r \sigma_I(r)}=m_{\sigma_I}.\]
The \emph{matching field polytope} $P_\Lambda$ is the convex hull of the exponent vectors of the $m_{\sigma_I}$, as $I$ ranges over all size $r$ subsets. 

If there is a grading on $\CC[x_{ij}]$ such that the chosen monomial is the unique leading term of $p_I$ for each $I$, then this matching field is called \emph{coherent}. 

Coherent matching fields are of most interest when they define toric degenerations. That is, if the $p_I$ are a SAGBI basis for the Cox ring of the Grassmannian under the grading, then there is a toric degeneration of the Grassmannian to the toric variety $X_{F_\Lambda}$ where $F_\Lambda$ is the normal fan to the matching field polytope $P_\Lambda$. 

In increasing degrees of generality, the papers cited above have identified families of matching fields that produce toric degenerations of $\Gr(r,n)$. One such family is the \emph{2-block diagonal} matching fields. Since we only consider $2$-block matching fields, we simply call them block diagonal matching fields. 
 We recall the definition here.
\begin{mydef}[\cite{mohammadi2}] Define a matching field $\Lambda_b$ for $\Gr(r,n)$ for each $0 \leq b \leq n$. The definition of $\Lambda_b$ is given by
\[\Lambda_b(J):=
\begin{cases} 
    (12) & |J \cap \{1,\dots,b\}| =1 \\
    \id & \text{otherwise} 
       \end{cases}\]
       where $J$ ranges over size $r$ subsets of $\{1,\dots,n\}$. 
\end{mydef}
Note that if $b=0$, this is the diagonal matching field which gives rise to the Gelfand--Cetlin toric degeneration. 
\begin{thm}[\cite{mohammadi2}]
The block diagonal matching field $\Lambda_b$  defines a toric degeneration of $\Gr(r,n)$ for any $b$. 
\end{thm}

\subsection{Generalizing to Kronecker moduli spaces}
Much like the Grassmannian, the semi-invariants $f_{T^\pm}$ of $\K$ are signed sums over (products of) symmetric groups.   It is therefore natural to extend the definition of matching fields from Grassmannians to Kronecker moduli space as a map
\[\{\text{linked pairs } T^\pm \text{ of weight } (-a r_2, a r_1) \} \mapsto W^\alpha=\Sym(r_1)^{a r_2} \times \Sym(r_2)^{a r_1},\]
for all $a$. 

Again, just as for Grassmannians, we can define \emph{coherent} matching fields for Kronecker quivers to be those where there exists a grading on $\CC[x^i_{jk}]$ such that, for each pair $T^\pm$, there is a unique monomial with maximal grading, and this monomial is the one associated to the element of $W^\alpha$ assigned by the matching field. 

Fix a Kronecker quiver $\K$, and a grading on $\CC[x^i_{jk}]$ given by a tuple of matrices of non-negative integers $[c^i_{jk}] \in M$. Note that this grading gives a grading on the coordinates of  $\Gr(r_1,nr_2)$ and $\Gr(r_2,nr_1)$. We'll be interested in the situation when $[c^i_{jk}]$ induces matching fields simultaneously on $\Gr(r_1,n r_2)$, $\Gr(r_2, n r_1)$ and $\K$, which we'll denote $\Lambda_1,$ $\Lambda_2$ and $\Lambda$ respectively.

Then the following is a natural question: is there a subset of linked tableaux pairs $T^\pm$ such that the semi-invariants $f_{T^\pm}$ are a finite SAGBI basis under this grading? In the following theorem, we give sufficient conditions for this to hold.    

\begin{thm} \label{thm:MF} Let  $[c^i_{jk}] \in M$ be a grading giving matching fields $\Lambda_1,$ $\Lambda_2$ and $\Lambda$ on $\Gr(r_1,nr_2)$, $\Gr(r_2,nr_1)$, and $\K$ respectively.   Suppose the following two conditions hold:
\begin{enumerate}
\item $\Lambda_1$ and $\Lambda_2$ induce toric degenerations on $\Gr(r_1,n r_2)$ and $\Gr(r_2, nr_1)$. 
\item Let $C_i$ be the cone over the matching field polytope $P_{\Lambda_i}$. Let $v_1,\dots,v_l$ be lattice generators for $C_1 \cap C_2$, and suppose that for each $i$ there exists some semi-invariant $f_i$ of $\K$ such that  the exponent vector of $\LM(f_i)$ is $v_i$.
\end{enumerate}
Then there is a toric degeneration of $\K$ to the normal toric variety $X_F$, where $F$ is the normal fan to the polytope 
\[ r_2 P_{\Lambda_1} \cap r_1 P_{\Lambda_2}.\]
\end{thm}
\begin{proof}
Suppose we have a grading satisfying the conditions of the theorem. Then as the Pl\"ucker coordinates are a SAGBI basis for the coordinate ring of each Grassmannian, the cone $C_1$ is the cone associated to  the algebra generated by $\{\LM(f): f \in \Cox(\Gr(r_1,n r_2))\},$ and the cone $C_2$ is the cone associated to  the algebra generated by $\{\LM(f): f \in \Cox(\Gr(r_2,n r_1))\}.$ Note that $C_1$ and $C_2$ are polyhedral cones, and hence so is $C_1 \cap C_2$, so in particular it is finitely generated. So a finite set $v_1, \dots, v_l$ as mentioned in the theorem statement exists. 

Our aim is now to show that there is a finite SAGBI basis of the Kronecker quiver under this grading. Suppose first $f$ is a semi-invariant of $\K$. Then $f$ is an element of the Cox ring of both Grassmannians, and therefore the exponent vector of $\LM(f)$ is in $C_1 \cap C_2.$ Conversely, suppose we have a lattice element $v \in C_1 \cap C_2$. Then we can write 
\[v=\sum a_i v_i.\]
Set $f=\prod f_i^{a_i}$. Then the exponent of the leading monomial of $f$ is $v$. We conclude that the lattice points of $C_1 \cap C_2$ are 
 exactly the exponent vectors of the leading monomials of semi-invariants of $\K$, and $f_1,\dots,f_l$ is a SAGBI basis for $\K$ with this grading. The remainder of the statement of the theorem follows from Corollary \ref{cor:toric}.
\end{proof}
\subsection{Applying Theorem \ref{thm:MF}: canonical(ish) tableaux}
We now explain how to use the theorem in practice; we'll apply this method in several examples in \S \ref{sec:examples}. The main point is that the construction of semi-invariants from double tableaux lends itself well to  checking the second condition in the theorem. 

Recall that in the Gelfand--Cetlin set up, for each element $v \in C_1 \cap C_2$, although there are many linked pairs $T^\pm$ such that $\Mon(T^\pm)$ has exponent $v$, there is a unique semi-standard linked pair giving this condition. Moreover, $\LM(f_{T^\pm})=\Mon(T^\pm)$. One major difference for matching field tableaux (linked or even just in the Grassmannian context) is that there is not, in general, a notion of ``semi-standard'' or canonical.

The desired notion  of canonical tableaux for a Grassmannian $\Gr(r,n)$ together with a coherent matching field would be a combinatorial condition on tableaux such that:
\begin{enumerate}
\item There is a unique canonical tableau with $\Mon(T)=m$ for every $m \in \{\LM(f): f \in \CC[\Gr(r,n)]\}$. 
\item For each column $T'$ in a canonical tableau $T$, the associated Pl\"ucker coordinate has leading monomial $\Mon(T')$. 
\end{enumerate}
Note that defining canonical tableaux for Grassmannians implies a construction for Kronecker moduli. Given a coherent matching field for $\K$, if we have defined canonical tableaux for the associated coherent matching fields of $\Gr(r_1,n r_2)$ and $\Gr(r_2,nr_1)$ then we define a canonical linked pair $(T^\pm)$ to be a linked pair such that both $T^+$ and $T^-$ are canonical. Then for each $v \in C_1 \cap C_2$, there is a unique canonical linked pair $T^\pm$ such that $\Mon(T^\pm)$ has exponent $v$. The other desired condition -- that $\LM(f_{T^\pm})=\Mon(T^\pm)$ for canonical pairs -- is not obviously (to the author) true. Instead, we check it as part of the construction in examples. 

Although canonical tableaux for matching fields have not been defined, we can use the following (dependent on choices) definition. This is what is used in \S \ref{sec:examples}, where it is successful in computations. Fix a coherent matching field $\Lambda$ on $\Gr(r,n)$ such that the Pl\"ucker coordinates are a SAGBI basis. Each of the $\binom{n}{r}$ Pl\"ucker coordinates $p_J$ has a leading term determined by the matching field, to which we can associate a $r \times 1$ tableau $C_J$ (there are no choices since all rows have length one).  We order the columns $C_J$ using the lex ordering given by writing the labels of the column as a word, so we have $J_1,\dots,J_{\binom{n}{r}}$. Given a leading monomial $m$ of some element $f \in \CC[\Gr(r,n)]$, there is at least one way  (but certainly finitely many ways) to write 
\[m=\prod_{i=1}^{\binom{n}{r}} \LM(p_{J_i})^{a_{J_i}}, a_{J_i} \geq 0.\]
Each choice of $a=(a_{J_i})$ corresponds to a tableau $T_a$ given by taken $a_{J_1}$ columns of $C_{J_1}$, then $a_{J_2}$ columns of $C_{J_2}$ and so forth. 
\begin{mydef}
The $\Lambda$-canonical tableau associated to $m$ to is the tableau $T_a$ given by taking the lex maximal choice of $a=(a_{J_i})$ for all $a$ satisfying 
\[m=\prod_{i=1}^{\binom{n}{r}} \LM(p_{J_i})^{a_{J_i}}, a_{J_i} \geq 0.\]
A tableau is said to be canonical if there exists an element $f\in \CC[\Gr(r,n)]$ such that $T$ is the $\Lambda$-canonical tableau associated to $\LM(f)$. 
\end{mydef}
\begin{rem}This definition is not meant to be definitive; as mentioned, it depends on choices (ordering of the $C_J$, ordering of the $a$); moreover, it is not combinatorial in the following sense. Given a tableau $T$, one can easily determine whether it is semi-standard by looking at it; this is not obvious for canonical tableau defined in this sense. \emph{Canonical-type tableau} are discussed in \cite{clarke2024minimalcellularresolutionspowers}, which may also be useful, although the definition above is more suitable for computations for now. 
\end{rem}
We now describe how to apply Theorem \ref{thm:MF}. For examples, see \S \ref{sec:examples}; these were computed using MAGMA.
\begin{enumerate}
\item Start by finding a grading $[c^i_{jk}]$ such that it induces matching fields on $\Gr(r_1,n r_2)$ and $\Gr(r_2,n r_1)$ for which the Pl\"ucker coordinates are a SAGBI basis in this order (for families where this is known, see \cite{mohammadi1, mohammadi2}). 
\item Compute lattice generators $v_1,\dots,v_k$ of the cone $C_1 \cap C_2$. 
\item Each generator $v_i$ gives a monomial $m_i$.  Using basic cone computations in $C_1$ and $C_2$, find the canonical linked tableau pair $T^\pm_i$ such that $\Mon(T^\pm_i)=m_i.$ 
\end{enumerate}
If each $T^\pm_i$ satisfies 
\[\LM(f_{T_i^\pm})=m_i.\]
then by Theorem \ref{thm:MF}, then set $\{f_{T_i^\pm}\}$ is a SAGBI basis for $\K$. 
If this condition fails (which did not happen in the examples computed), then one can try the following variation on the construction, beginning at Step 3:
\begin{enumerate}\addtocounter{enumi}{2}
\item For each $m_i$, search over the linked pairs of tableaux $T^\pm$ obtained from $T^\pm_i$ (given above) by permuting the rows in all possible ways, and find one (if possible)  such
\[\LM(f_{T^\pm})=m_i.\]
Call this tableaux pair $S^\pm_i.$
\end{enumerate}
Again by Theorem \ref{thm:MF}, then set $\{f_{S_i^\pm}\}$ is a SAGBI basis for $\K$.

We will conclude this section by describing a family of gradings that satisfy  the first step above. Fix a Kronecker moduli space $\K$. For a tuple $p:=(p_1,\dots,p_{r_2})$ of integers, define the $r_2 \times r_1$ matrix
\[C_p:=\begin{bmatrix}
p_1 & 2 p_1 & \cdots & r_1 p_1\\
p_2 & 2 p_2 & \cdots & r_1 p_2\\
\vdots & & & \vdots \\
p_{r_2} & 2 p_{r_2} & \cdots & r_1 p_{r_2}\\
\end{bmatrix}.\]
Set $p^i:=(((r_2+1)^{i-1}, 2 (r_2+1)^{i-1}, \dots, r_2 (r_2+1)^{i-1}).$

Finally, we set 
\[C_0:=(C_{p^1},\dots,C_{p^n}).\]
\begin{lem}\label{lem:0grading} The leading term of a Pl\"ucker coordinate of $\Gr(r_1,nr_2)$ under the grading induced by $C_0$ is the diagonal term.
\end{lem}
\begin{proof} The induced grading is given by vertically adjoining the $n$ matrices in $C_0$ into a single $n r_2 \times r_1$ matrix. For simplicity, we won't transpose, but just take minors given by size $r_1$ subsets of rows. Note that for each row, the difference between consecutive entries is constant. This sequence of differences as one moves down the rows is the concatenation of the $p^i$; this concatenation is a strictly increasing sequence. This is clear within each $p^i$; for the rest, note that the last term of $p^{i}$ is $r_2 (r_2+1)^{i-1}$, and the first term of $p^{i+1}$ is $(r_2+1)^i$, which is strictly larger. 

Consider a Pl\"ucker coordinate of $\Gr(r_1,nr_2)$ determined by a choice of $r_1$ rows of a generic $n r_2 \times r_1$ matrix. That is, this is given by some choice of double indices
\[(i_1,j_1) < \dots <(i_l,j_l),\]
where $<$ is the lex ordering. Then the grading of the monomial corresponding to $\sigma \in \Sym(r_1)$ is
\[ \sum_{t=1}^{r_1} \sigma(t) p^{i_t}_{j_t},\]
and as $p^{i_1}_{j_1}< \cdots<p^{i_{r_1}}_{j_{r_1}},$ this achieves a unique maximum when $\sigma$ is the identity permutation. 
\end{proof}
\begin{lem} \label{lem:0grading1} The leading term of a Pl\"ucker coordinate of $\Gr(r_2,nr_1)$ under the grading induced by $C_0$ is the diagonal term.
\end{lem}
\begin{proof} The induced grading is given by horizontally  adjoining the $n$ matrices in $C_0$ into a single $ r_2 \times n r_1$ matrix. The argument is analogous to Lemma \ref{lem:0grading}. For each column, the difference between consecutive entries in constant. Going from left to right through the columns, these differences are the concatenation of $q^1,\dots,q^n$, where
\[q^i:=((r_2+1)^i, \dots, r_2(r_2+1)^i).\]
As $r_1 \leq r_2$, the concatenation is a strictly increasing sequence. 
Consider a Pl\"ucker coordinate of $\Gr(r_1,nr_2)$ determined by a choice of $r_2$ columns of a generic $ r_2 \times n r_1$ matrix. That is, this is given by some choice of double indices
\[(i_1,j_1) < \dots <(i_l,j_l),\]
where $<$ is the lex ordering. Then the grading of the monomial corresponding to $\sigma \in \Sym(r_2)$ is
\[ \sum_{t=1}^{r_2} \sigma(t) q^{i_t}_{j_t},\]
and as $q^{i_1}_{j_1}< \cdots<q^{i_{r_1}}_{j_{r_1}},$ this achieves a unique maximum when $\sigma$ is the identity permutation. 

\end{proof}

\begin{prop} The grading $C_0$ induces the Gelfand--Cetlin toric degeneration of $\K$.
\end{prop}
\begin{proof}
Basically, we are checking the conditions of Theorem \ref{thm:MF}.  By Lemmas \ref{lem:0grading} and \ref{lem:0grading1}, the first condition holds, as the associated matching field is just the diagonal matching field, which corresponds to the Gelfand--Cetlin degeneration. Proposition 5.5 from \cite{HeubergerKalashnikov} is the precise statement of the second condition. Therefore, by the theorem, we obtain a toric degeneration of $\K$ to what we have defined as the Gelfand--Cetlin degeneration. 
\end{proof} 
It would be interesting to find families of weights that induce block diagonal matching fields on both associated Grassmannians. In the section below, we explain how to do this when both $b$ and $r_1$ are equal to $2$. 

\section{Examples} \label{sec:examples}
We describe the polytopes obtained in some small examples of Kronecker moduli spaces. 
\begin{eg}[$K^{3}_{2 3}$] \label{eg:K323}The cone $C^3_{2 3}$ is generated in degree $1$ by 20 lattice generators, corresponding to 20 primitive semi-standard Young tableaux pairs. These are listed in \cite{HeubergerKalashnikov}. The slice at height 1 is a polytope with $18$ vertices, and the dual fan has 13 rays. The associated toric variety is Gorenstein Fano with terminal singularities.
\end{eg}
\begin{eg}[$K^4_{2 3}$]\label{eg:K423} In \cite{HeubergerKalashnikov}, we did not compute a full SAGBI basis of Cox ring of $K^4_{2 3}$, a 12 dimensional variety. Using the method outlined above, we can now determine the cone $C^4_{2 3}$. This has 232 lattice generators, corresponding to 232 primitive semi-standard Young tableaux pairs. There are 126 in degree 1, 86 in degree 2, and 20 in degree 3. We record the 20 in the highest degree here:
\[   \begin{bmatrix}
         
            1 1 &
            1 1 &
            1 1 &
            1 1 &
            1 1 &
            2 2 \\

            1 2 &
            1 2 &
            1 2 &
            1 2 &
            3 2 &
            3 2 \\

            2 1 &
            2 1 &
            2 1 &
            3 1 &
            4 2 &
            4 2 \\
         
    \end{bmatrix},\hspace{5mm}
    \begin{bmatrix}
         
            1 1 &
            1 1 &
            1 1 &
            1 1 &
            1 1 &
            2 2 \\

            1 2 &
            1 2 &
            1 2 &
            2 2 &
            3 2 &
            3 2 \\

            2 1 &
            2 1 &
            2 1 &
            3 1 &
            4 2 &
            4 2 \\
         
    \end{bmatrix},\hspace{5mm}
    \begin{bmatrix}
         
            1 1 &
            1 1 &
            1 1 &
            1 1 &
            1 1 &
            2 2 \\

            1 2 &
            1 2 &
            1 2 &
            1 2 &
            3 2 &
            3 2 \\

            2 1 &
            2 1 &
            2 1 &
            2 1 &
            4 2 &
            4 2 \\
         
  \end{bmatrix},\]
  
  \[\begin{bmatrix}
         
            1 1 &
            1 1 &
            2 2 &
            2 2 &
            3 2 &
            3 2 \\

            2 1 &
            2 1 &
            3 1 &
            3 1 &
            4 1 &
            4 1 \\

            3 1 &
            3 2 &
            4 2 &
            4 2 &
            4 2 &
            4 2 \\
         
  \end{bmatrix},\hspace{5mm}
  \begin{bmatrix}
         
            1 1 &
            1 1 &
            2 2 &
            2 2 &
            3 2 &
            3 2 \\

            2 1 &
            2 1 &
            3 1 &
            3 1 &
            4 1 &
            4 1 \\

            3 1 &
            4 2 &
            4 2 &
            4 2 &
            4 2 &
            4 2 \\
         
        \end{bmatrix},\hspace{5mm}
        \begin{bmatrix}
         
            1 1 &
            1 1 &
            2 2 &
            2 2 &
            3 2 &
            3 2 \\

            2 1 &
            2 1 &
            3 1 &
            3 1 &
            4 1 &
            4 1 \\

            3 1 &
            3 2 &
            3 2 &
            4 2 &
            4 2 &
            4 2 \\
         
        \end{bmatrix},\]
    \[    \begin{bmatrix}
         
            1 1 &
            1 1 &
            1 1 &
            1 1 &
            1 1 &
            2 2 \\

            1 2 &
            1 2 &
            1 2 &
            1 2 &
            3 2 &
            3 2 \\

            2 1 &
            2 1 &
            3 1 &
            3 1 &
            4 2 &
            4 2 \\
         
        \end{bmatrix},\hspace{5mm}
        \begin{bmatrix}
         
            1 1 &
            1 1 &
            1 1 &
            1 1 &
            1 1 &
            2 2 \\

            1 2 &
            1 2 &
            2 2 &
            2 2 &
            3 2 &
            3 2 \\

            2 1 &
            2 1 &
            3 1 &
            3 1 &
            4 2 &
            4 2 \\
         
        \end{bmatrix},\hspace{5mm}
        \begin{bmatrix}
         
            1 1 &
            1 1 &
            1 1 &
            1 1 &
            1 1 &
            2 2 \\

            1 2 &
            1 2 &
            1 2 &
            2 2 &
            3 2 &
            3 2 \\

            2 1 &
            2 1 &
            3 1 &
            3 1 &
            4 2 &
            4 2 \\
         
        \end{bmatrix},\]
        \[\begin{bmatrix}
         
            1 1 &
            1 1 &
            2 2 &
            2 2 &
            3 2 &
            3 2 \\

            2 1 &
            2 1 &
            4 1 &
            4 1 &
            4 1 &
            4 1 \\

            3 1 &
            4 2 &
            4 2 &
            4 2 &
            4 2 &
            4 2 \\
         
        \end{bmatrix},\hspace{5mm}
        \begin{bmatrix}
         
            1 1 &
            1 1 &
            2 2 &
            3 2 &
            3 2 &
            3 2 \\

            2 1 &
            2 1 &
            4 1 &
            4 1 &
            4 1 &
            4 1 \\

            3 1 &
            4 2 &
            4 2 &
            4 2 &
            4 2 &
            4 2 \\
         
        \end{bmatrix},\hspace{5mm}
        \begin{bmatrix}
         
            1 1 &
            1 1 &
            3 2 &
            3 2 &
            3 2 &
            3 2 \\

            2 1 &
            2 1 &
            4 1 &
            4 1 &
            4 1 &
            4 1 \\

            3 1 &
            4 2 &
            4 2 &
            4 2 &
            4 2 &
            4 2 \\
         
        \end{bmatrix},\]
       \[ \begin{bmatrix}
         
            1 1 &
            1 1 &
            1 1 &
            1 1 &
            2 1 &
            2 2 \\

            1 2 &
            1 2 &
            2 2 &
            2 2 &
            3 2 &
            3 2 \\

            2 1 &
            2 1 &
            3 1 &
            3 1 &
            4 2 &
            4 2 \\
         
        \end{bmatrix},\hspace{5mm}
        \begin{bmatrix}
         
            1 1 &
            1 1 &
            1 1 &
            1 1 &
            2 1 &
            2 2 \\

            1 2 &
            1 2 &
            1 2 &
            2 2 &
            3 2 &
            3 2 \\

            2 1 &
            2 1 &
            3 1 &
            3 1 &
            4 2 &
            4 2 \\
         
        \end{bmatrix},\hspace{5mm}
        \begin{bmatrix}
         
            1 1 &
            1 1 &
            1 1 &
            2 1 &
            2 1 &
            2 2 \\

            1 2 &
            1 2 &
            2 2 &
            2 2 &
            3 2 &
            3 2 \\

            2 1 &
            2 1 &
            3 1 &
            3 1 &
            4 2 &
            4 2 \\
         
        \end{bmatrix},\]
      \[  \begin{bmatrix}
         
            1 1 &
            1 1 &
            1 1 &
            1 1 &
            2 1 &
            2 2 \\

            1 2 &
            1 2 &
            1 2 &
            2 2 &
            3 2 &
            3 2 \\

            2 1 &
            2 1 &
            2 1 &
            3 1 &
            4 2 &
            4 2 \\
         
        \end{bmatrix},\hspace{5mm}
        \begin{bmatrix}
         
            1 1 &
            1 1 &
            2 2 &
            3 2 &
            3 2 &
            3 2 \\

            2 1 &
            2 1 &
            3 1 &
            4 1 &
            4 1 &
            4 1 \\

            3 1 &
            3 2 &
            4 2 &
            4 2 &
            4 2 &
            4 2 \\
         
        \end{bmatrix},\hspace{5mm}
        \begin{bmatrix}
         
            1 1 &
            1 1 &
            2 2 &
            2 2 &
            3 2 &
            3 2 \\

            2 1 &
            2 1 &
            3 1 &
            4 1 &
            4 1 &
            4 1 \\

            3 1 &
            3 2 &
            4 2 &
            4 2 &
            4 2 &
            4 2 \\
         
        \end{bmatrix},\]
   \[     \begin{bmatrix}
         
            1 1 &
            1 1 &
            2 2 &
            2 2 &
            3 2 &
            3 2 \\

            2 1 &
            2 1 &
            3 1 &
            4 1 &
            4 1 &
            4 1 \\

            3 1 &
            4 2 &
            4 2 &
            4 2 &
            4 2 &
            4 2 \\
         
        \end{bmatrix},\hspace{5mm}
        \begin{bmatrix}
         
            1 1 &
            1 1 &
            2 2 &
            3 2 &
            3 2 &
            3 2 \\

            2 1 &
            2 1 &
            3 1 &
            4 1 &
            4 1 &
            4 1 \\

            3 1 &
            4 2 &
            4 2 &
            4 2 &
            4 2 &
            4 2 \\
         
    \end{bmatrix}.\] 
 The polytope which is the height one slice of $C^4_{2 3}$ has 141 vertices. The dual fan has 26 rays, and the associated toric variety is Fano, but not Gorenstein. The polytope spanned by the rays is complicated; conjecturally, a maximally mutable Laurent polynomial supported on this polytope should be a mirror to $K^4_{23}$, but computing such a polynomial on such a large dimensional polytope is very hard.
\end{eg}

We now turn to considering examples of toric degenerations of Kronecker quivers coming from matching fields. In the special case where $r_1=2$, we show that there is a grading $C_2$ giving the block diagonal matching field with $b=2$. We first define $C_2$. 
 
\begin{mydef} Adjoin the $n$ matrices in $C_0$ into a single $ r_1 \times n r_2$ matrix $B_0$, and let $B_2$ be the matrix obtained replacing the first two rows of $B_0$ with the first two rows of $C_{p^{n+1}}$.
 Breaking this large matrix into $n$ small matrices, we obtain a new grading which we denote $C_2.$
\end{mydef}
\begin{eg} If $r_1=2,r_2=3$ and $n=3$, then $C_0$ is
\[(\begin{bmatrix}
 1 &2 \\ 
 2 &4 \\
 3 & 6\\
\end{bmatrix},\hspace{3mm}
\begin{bmatrix}
4& 8\\
8& 16\\ 
12& 24\\ 
\end{bmatrix},\hspace{3mm}
\begin{bmatrix}
16& 32\\
 32& 64\\ 
 48& 96\\
\end{bmatrix}).\]
The new grading, $C_2$ is
\[(\begin{bmatrix}
 64 &128 \\ 
 128 &256 \\
 3 & 6\\
\end{bmatrix},\hspace{3mm}
\begin{bmatrix}
4& 8\\
8& 16\\ 
12& 24\\ 
\end{bmatrix},\hspace{3mm}
\begin{bmatrix}
16& 32\\
 32& 64\\ 
 48& 96\\
\end{bmatrix}).\]

\end{eg}

\begin{lem}\label{lem:kgrading} Assume that $r_1=2$. Then the leading term of a Pl\"ucker coordinate of $\Gr(2,nr_2)$ under the grading induced by $C_2$ agrees with the $b=2$ block diagonal matching field.
\end{lem} 
\begin{proof}
Consider the Pl\"ucker coordinate $p_{ij}$ where $i<j$. If $i>2$, or $i=1,j=2$ then the grading is the same as in the $C_0$ case, and by the same argument the leading term is the diagonal. If $\{i,j\} \cap \{1,2\}$ is a set of size $1$, then the grading of the minor is the same as for some minor in $C_0$ for $n+1$, but with the rows swapped. Therefore the leading term is the anti-diagonal, so the matching field is $(1 2)$ as claimed.
\end{proof}

\begin{lem}\label{lem:kgrading1} Assume that $r_1=2$. Then the leading term of a Pl\"ucker coordinate of $\Gr(r_2,2n)$ under the grading induced by $C_2$ agrees with the $b=2$ block diagonal matching field.
\end{lem} 
\begin{proof} Consider a Pl\"ucker coordinate $p_{i_1 \dots i_{r_2}}$, $i_1<\cdots<i_{r_2}$. If $i_1>2$, then the grading is the same as in $C_0$ case, and so the leading term is the diagonal as required. 

Consider the case when $i_1=1$ and $i_2=2$. Every monomial in the minor has a contribution from the first two columns and from the remaining $r_2-2$ columns indexed by $i_{3},\dots,i_{r_2}$. If we look at all possible contributions from the first two columns, the maximally graded option is upper two diagonal terms.  Similarly, looking at contributions from the last $r_2-2$ columns, as the entries of the columns increase, the maximally graded contribution comes from the last $r_2-2$ rows. Now by the same argument as in Lemma \ref{lem:0grading1}, the maximally graded monomial is thus the diagonal term. 

The final case to consider is when either $\{i_1,i_2\} \cap \{1,2\}$ is a set of size $1$. As in the previous case, we can consider the first column separately from the remaining.  The second entry in the $i_1$ column is the maximal one, and it is at least $(r_2+1)^{n}$ larger than any other entry in the column. Now looking at the last $r_2-1$ chosen columns, the maximal grading achieved by a monomial in sub-minor is the diagonal term in the bottom oriented minor. Swapping the variable from the $i_2$ column from the second row to the first row reduces it by $(r_2+1)^k$ for $k<n$, so the minor corresponding to the matching field $(1 2)$ is the leading term as required. 
\end{proof}

We now apply Theorem \ref{thm:MF} to the $C_2$ grading in the examples $K^3_{2 3}$ and $K^4_{2 3}$. 
\begin{eg}\label{eg:MFK323} In this example, we consider the matching field on $K^3_{2 3}$ given by the $C_2$ grading. Since $C_2$ induces the $b=2$ block diagonal matching field for both $\Gr(2,9)$ and $\Gr(3,6)$ (by Lemmas \ref{lem:kgrading} and \ref{lem:kgrading1}), the first condition follows by \cite{mohammadi2}. 

To apply the theorem, compute the cones $C_1$ and $C_2$ and their intersection $C$. The cone $C$ is generated by 20 lattice generators. This is a cone generated by 20 lattice generators, all in degree 1. For each generator $v$, let $T^\pm_v$ be the canonical linked pair for the monomial associated to $v$. 
The 20 $T^+_v$ are
\[\begin{bmatrix}

            2 1 &
            2 1 \\

            2 2 &
            3 1 \\

            3 2 &
            3 2 \\
        
    \end{bmatrix},\hspace{5mm}
    \begin{bmatrix}
         
            2 1 &
            3 1 \\

            3 1 &
            1 2 \\

            3 2 &
            3 2 \\
        
    \end{bmatrix}, \hspace{5mm}
    \begin{bmatrix}
         
            2 1 &
            3 1 \\

            2 2 &
            1 2 \\

            3 1 &
            3 2 \\
        
    \end{bmatrix}, \hspace{5mm}
    \begin{bmatrix}
         
            2 1 &
            2 1 \\

            1 2 &
            3 1 \\

            3 2 &
            3 2 \\
        
    \end{bmatrix}, \hspace{5mm}
    \begin{bmatrix}
         
            2 1 &
            2 1 \\

            1 2 &
            3 1 \\

            2 2 &
            3 2 \\
        
    \end{bmatrix}, \]
  \[  \begin{bmatrix}
         
            2 1 &
            2 1 \\

            1 2 &
            2 2 \\

            3 1 &
            3 2 \\
        
    \end{bmatrix}, \hspace{5mm}
    \begin{bmatrix}
         
            2 1 &
            2 1 \\

            1 2 &
            2 2 \\

            2 2 &
            3 1 \\
        
    \end{bmatrix}, \hspace{5mm}
    \begin{bmatrix}
         
            2 1 &
            3 1 \\

            1 2 &
            1 2 \\

            3 1 &
            3 2 \\
        
    \end{bmatrix}, \hspace{5mm}
    \begin{bmatrix}
         
            2 1 &
            2 1 \\

            1 2 &
            1 2 \\

            3 1 &
            3 2 \\
        
    \end{bmatrix}, \hspace{5mm}
    \begin{bmatrix}
         
            2 1 &
            2 1 \\

            1 2 &
            1 2 \\

            2 2 &
            3 1 \\
        
    \end{bmatrix}, \]
    \[\begin{bmatrix}
         
            1 1 &
            2 1 \\

            1 2 &
            3 1 \\

            3 2 &
            3 2 \\
        
    \end{bmatrix}, \hspace{5mm}
    \begin{bmatrix}
         
            1 1 &
            2 1 \\

            1 2 &
            3 1 \\

            2 2 &
            3 2 \\
        
    \end{bmatrix}, \hspace{5mm}
   \begin{bmatrix}
         
            1 1 &
            2 1 \\

            1 2 &
            2 2 \\

            2 1 &
            3 2 \\
        
    \end{bmatrix}, \hspace{5mm}
    \begin{bmatrix}
         
            1 1 &
            3 1 \\

            1 2 &
            1 2 \\

            3 1 &
            3 2 \\
        
    \end{bmatrix}, \hspace{5mm}
    \begin{bmatrix}
         
            1 1 &
            3 1 \\

            1 2 &
            1 2 \\

            2 1 &
            3 2 \\
        
    \end{bmatrix}, \]
   \[ \begin{bmatrix}
         
            1 1 &
            2 1 \\

            1 2 &
            1 2 \\

            3 1 &
            3 2 \\
        
    \end{bmatrix}, \hspace{5mm}
    \begin{bmatrix}
         
            1 1 &
            2 1 \\

            1 2 &
            1 2 \\

            2 2 &
            3 1 \\
        
    \end{bmatrix}, \hspace{5mm}
    \begin{bmatrix}
         
            1 1 &
            2 1 \\

            1 2 &
            1 2 \\

            2 1 &
            3 2 \\
        
    \end{bmatrix}, \hspace{5mm}
    \begin{bmatrix}
         
            1 1 &
            2 1 \\

            1 2 &
            1 2 \\

            2 1 &
            2 2 \\
        
    \end{bmatrix}, \hspace{5mm}
    \begin{bmatrix}
         
            1 1 &
            1 1 \\

            1 2 &
            1 2 \\

            2 1 &
            3 2 \\
        
    \end{bmatrix}.\]
The $T^-_v$, in the same order (so that the first tableau above is in a linked pair with the first tableau below, and so forth), are

  \[  \begin{bmatrix}
        
            2 1 &
            2 1 &
            3 2 \\

            2 2 &
            3 3 &
            3 3 \\
        
    \end{bmatrix},\hspace{5mm}
    \begin{bmatrix}
        
            2 1 &
            3 1 &
            3 2 \\

            1 2 &
            3 3 &
            3 3 \\
        
    \end{bmatrix},\hspace{5mm}
    \begin{bmatrix}
        
            2 1 &
            3 1 &
            3 3 \\

            2 2 &
            3 3 &
            1 2 \\
        
    \end{bmatrix},\hspace{5mm}
    \begin{bmatrix}
        
            2 1 &
            2 1 &
            3 2 \\

            1 2 &
            3 3 &
            3 3 \\
        
    \end{bmatrix},\hspace{5mm}
    \begin{bmatrix}
        
            2 1 &
            2 1 &
            3 2 \\

            1 2 &
            2 3 &
            3 3 \\
        
    \end{bmatrix},\]
   \[ \begin{bmatrix}
        
            2 1 &
            2 1 &
            3 3 \\

            2 2 &
            3 3 &
            1 2 \\
        
    \end{bmatrix}, \hspace{5mm}
   \begin{bmatrix}
        
            2 1 &
            2 1 &
            3 3 \\

            2 2 &
            2 3 &
            1 2 \\
        
    \end{bmatrix},\hspace{5mm}
    \begin{bmatrix}
        
            2 1 &
            3 1 &
            3 3 \\

            1 2 &
            3 3 &
            1 2 \\
        
    \end{bmatrix},\hspace{5mm}
    \begin{bmatrix}
        
            2 1 &
            2 1 &
            3 3 \\

            1 2 &
            3 3 &
            1 2 \\
        
    \end{bmatrix},\hspace{5mm}
    \begin{bmatrix}
        
            2 1 &
            2 1 &
            3 3 \\

            1 2 &
            2 3 &
            1 2 \\
        
    \end{bmatrix},\]
  \[  \begin{bmatrix}
        
            1 1 &
            2 1 &
            3 2 \\

            1 2 &
            3 3 &
            3 3 \\
        
    \end{bmatrix},\hspace{5mm}
    \begin{bmatrix}
        
            1 1 &
            2 1 &
            3 2 \\

            1 2 &
            2 3 &
            3 3 \\
        
    \end{bmatrix},\hspace{5mm}
   \begin{bmatrix}
        
            1 1 &
            2 1 &
            2 3 \\

            1 2 &
            2 2 &
            3 3 \\
        
    \end{bmatrix},\hspace{5mm}
    \begin{bmatrix}
        
            1 1 &
            3 1 &
            3 3 \\

            1 2 &
            3 3 &
            1 2 \\
        
    \end{bmatrix},\hspace{5mm}
    \begin{bmatrix}
        
            1 1 &
            2 3 &
            3 1 \\

            1 2 &
            1 2 &
            3 3 \\
        
    \end{bmatrix},\]
 \[   \begin{bmatrix}
        
            1 1 &
            2 1 &
            3 3 \\

            1 2 &
            3 3 &
            1 2 \\
        
    \end{bmatrix},\hspace{5mm}
    \begin{bmatrix}
        
            1 1 &
            2 1 &
            3 3 \\

            1 2 &
            2 3 &
            1 2 \\
        
    \end{bmatrix},\hspace{5mm}
    \begin{bmatrix}
        
            1 1 &
            2 1 &
            2 3 \\

            1 2 &
            1 2 &
            3 3 \\
        
    \end{bmatrix},\hspace{5mm}
   \begin{bmatrix}
        
            1 1 &
            2 1 &
            2 3 \\

            1 2 &
            2 3 &
            1 2 \\
        
    \end{bmatrix},\hspace{5mm}
    \begin{bmatrix}
        
            1 1 &
            1 1 &
            2 3 \\

            1 2 &
            1 2 &
            3 3 \\
        
    \end{bmatrix}.\]
One can then check for each of the 20 generators $v$, that the exponent of the leading monomial of $f_{T^\pm_v}$ is $v$ as required. 

We therefore obtain that there is a toric degeneration of $K^3_{23}$ to the toric variety $X_F$, where $F$ is the normal fan of the polytope obtained by taking the height 1 slice of $C$. This polytope has 20 vertices (its only lattice points), which are the lattice generators of $C$. The toric variety is Fano, Gorenstein, and terminal. The fan $F$ has 12 rays. As in \cite{HeubergerKalashnikov}, we can write down a Laurent polynomial $f$ supported on the polytope $Q$, where $Q$ given by taking the convex hull of the primitive generators of the rays of $F$. The polytope $Q$ is reflexive and the only lattice points are its 12 vertices and the origin. 

 The Laurent polynomial $f$ is
\begin{align*} x_1 x_2 x_3/(x_5 x_6) + x_1 x_3/x_6 + x_1/x_5 + x_2 x_3 + x_2 x_4/x_5 \\+ 
x_3/(x_4 x_6) + x_4 + x_5 + x_6 + 1/(x_2 x_4) + x_5/(x_1 x_4) + 1/(x_1 x_3).\end{align*}
This Laurent polynomial mirror is expected to be a conjectural mirror to $K^3_{23}$. See \cite{fanomanifolds} or \cite{HeubergerKalashnikov} for a more careful definition. Computing the first 20 terms of the \emph{period sequence} of this Laurent polynomial, we obtain
\[(1, 0, 0, 18, 0, 0, 4590, 0, 0, 1728720, 0, 0, 876610350, 0, 0, 520461209268, 
0, 0, 343838539188144, 0, 0). \]
In \cite{HeubergerKalashnikov} we produced an analogue of this Laurent polynomial using the $b=0$ grading (recalled here in Example \ref{eg:K323}); while the underlying polytopes are \emph{not} isomorphic, the first 20 terms of the period sequence agree. This suggests (but does not prove) that the two Laurent polynomials are mutation equivalent; this would not be surprising, as \cite{mohammadi3} shows that the statement holds for each of the Grassmannians and their respective pairs of polytopes. In particular, since the two period sequences agree up to 20 terms, it follows from \cite{HeubergerKalashnikov} that the first 20 terms of period sequence of $f$ agree with the first 20 terms of the quantum period of $K^3_{23}$. 
\end{eg}

\begin{eg} In this example, we consider the matching field on $K^4_{2 3}$ given by the $C_2$ grading. Since $C_2$ induces the $b=2$ block diagonal matching field for both $\Gr(2,12)$ and $\Gr(3,8)$ (by Lemmas \ref{lem:kgrading} and \ref{lem:kgrading1}), the first condition follows by \cite{mohammadi2}. To apply the theorem, compute the cones $C_1$ and $C_2$ and their intersection $C$. Using MAGMA, we compute that the cone $C$ is generated by 206 lattice elements, $126$ in degree one, and $80$ in degree 2. Just as in the previous example, Example \ref{eg:MFK323}, to each lattice generator $v$, we associate the canonical pair of linked tableaux $T^\pm_v$. Using MAGMA, we check that for each of the 206 generators $v$, the exponent of the leading monomial of $f_{T^\pm_v}$ is $v$ as required by the theorem.

We therefore obtain that there is a toric degeneration of $K^4_{23}$ to the toric variety $X_F$, where $F$ is the normal fan of the polytope $P$ obtained by taking the height 1 slice of $C$. The polytope $P$ has 142 vertices, not all lattice points. The toric variety $X_F$ is Gorenstein, Fano, \emph{and} terminal; so much nicer than the corresponding $b=0$ toric variety in Example \ref{eg:K423}. Let $Q$ be the polytope which is the convex hull of the rays of $F$. Then $Q$ has 30 vertices; other than the origin, these are the only lattice points in $Q$. This allows us to write down a candidate Laurent polynomial:
\begin{align*}
x_1 x_2 x_3 x_4 x_7 x_8 x_9/(x_5 x_6 x_{11} x_{12}) + x_1 x_2 x_3 x_4 x_7 x_8/(x_5 x_6 x_{11})\\ + 
x_1 x_2 x_3 x_7 x_8/(x_5 x_6) + x_1 x_3 x_7 x_9/(x_6 x_{12}) + x_1 x_3 x_7/(x_5 x_6 x_{10} x_{11} x_{12})\\ + 
x_1 x_3 x_7/(x_5 x_6 x_9 x_{10} x_{11}) + x_1 x_3/(x_5 x_6 x_9 x_{11}) + x_1 x_3 x_7/(x_4 x_5 x_6 x_9 x_{10}) \\
+ x_1 x_3/(x_4 x_5 x_6 x_9) + x_1/x_5 + x_2 x_3 x_7 x_8 + x_2 x_4 x_8 x_{10}/(x_5 x_{11}) \\+ 
x_3 x_5 x_7 x_9/(x_4 x_6 x_8 x_{10} x_{12}) + x_3 x_5 x_7/(x_4 x_6 x_8 x_{10}) + x_3 x_5/(x_4 x_6 x_8) \\+ 
x_3 x_7/(x_4 x_9 x_{10}) + x_3/(x_4 x_9) + x_3/(x_4 x_6) + x_4 + x_5 + x_6 + x_7 + x_8 + x_9\\ + x_{10} 
+ x_{11} + x_{12} + 1/(x_2 x_4 x_9 x_{10}) + x_5 x_{11}/(x_1 x_4 x_8 x_{10}) + 1/(x_1 x_3 x_7 x_8)\end{align*}
We have not computed the quantum period of $K^4_{23}$, and so cannot compare that sequence with the period sequence of this Laurent polynomial. However, we note that the anti-canonical divisor of $X_F$ has index $4$ in the class lattice of $X_F$, and hence the period sequence should have the same zero pattern as the expected period sequence of $K^4_{23}$, which also has Fano index 4. 
\end{eg}

\bibliographystyle{amsplain}
\bibliography{epsrc}
\end{document}